\documentclass{article}

\usepackage{float}
\usepackage{color}
 \catcode`@=11
\usepackage{color}
\usepackage{amsthm,amssymb,amsfonts}
\usepackage{amsmath}
\usepackage{graphicx}
\usepackage{epstopdf}
\catcode`@=11 \@addtoreset{equation}{section}

\catcode`@=12

\newtheorem{theorem}{Theorem}[section]
\newtheorem{proposition}{Proposition}[section]
\newtheorem{lemma}{Lemma}[section]

\theoremstyle{definition}

\newtheorem{definition}{Definition}[section]

\newtheorem{assumptions}{Hypothesis}[section]

\usepackage[colorlinks=true,urlcolor=red,linkcolor=blue,citecolor=red,bookmarks=red]{hyperref}
\title{Indirect stabilization of nonlinear coupled wave equations in the presence of time delay or  indefinite damping}
\author{
	{\sc Alhabib Moumni}\\
MAMCS Group, FST Errachidia, \\Moulay Ismail University of Meknes,\\ P.O. Box 509,
Boutalamine, 52000, Errachidia, Morocco\\
	email: alhabibmoumni@gmail.com\\
		{\sc Cristina Pignotti}\\
Dipartimento di Ingegneria e Scienze dell’Informazione e Matematica, \\
Università degli Studi dell’Aquila,\\ Via Vetoio, 67100 L’Aquila, Italy\\
	email: cristina.pignotti@univaq.it\\
	{\sc Jawad Salhi}\\
	MAMCS Group, FST Errachidia, \\Moulay Ismail University of Meknes,\\ P.O. Box 509,
	Boutalamine, 52000, Errachidia, Morocco\\
	email: sj.salhi@gmail.com\\
	{\sc  Mouhcine Tilioua}
	\\
MAMCS Group, FST Errachidia, \\Moulay Ismail University of Meknes,\\ P.O. Box 509,
Boutalamine, 52000, Errachidia, Morocco\\
	email: m.tilioua@umi.ac.ma}
\date{}
\begin{document}
	
	\maketitle
	\begin{abstract}
This paper explores the exponential stability of two nonlinear wave equations coupled through their velocities. The analysis is divided into two main cases. First, we consider a system where one equation is damped, while the other experiences a discrete time delay. By reformulating the problem in an abstract framework, we use semigroup theory and energy methods to establish well-posedness and derive conditions that guarantee exponential energy decay. In the second case, we examine a scenario where a frictional damping term appears in the first equation, while the second equation contains an indefinite damping term, namely with a sign-changing coefficient. Although this setup can be viewed as a special case of the first, we analyze it separately and show that exponential stability still holds under a weaker condition.
	\end{abstract}
	
	Keywords: Stabilization, Nonlinear hyperbolic systems, delay feedback, indefinite damping
	
	MSC 2020:  93B07, 35L05, 93D15, 93C20
	%%%%%%%%%%%%%%%%%%%%%%%%%%%%%%%%%%%
\section{Introduction} \label{intro}

Time delays are an inherent feature of numerous real-world systems, particularly those where the current state is influenced by historical values. Consequently, their inclusion in mathematical models is essential for accurately representing phenomena in biology, mechanics, and engineering. 
 
Over the past decades, substantial research has focused on the stability of systems with delayed damping, a phenomenon often stemming from sources such as control force computation, signal transmission, or measurement processing.

%A key characteristic of time delays is their propensity to destabilize systems that would otherwise remain stable in the absence of delay.

It is well-established that time delays frequently induce instability in systems that are otherwise,  i.e. without any time lags,  uniformly asymptotically stable (see e.g. \cite{datko1991, nicaise2006stability}). This destabilizing effect has consequently motivated significant research efforts aimed at developing effective strategies to mitigate or control these delay effects; see for instance \cite{xu2006stabilization,pignotti2012note,nicaise2006stability,ammari2010feedback,adimy2003global,ammari2022well,continelli2025energy}.

In many practical applications, system dynamics are too complex to be described by a single scalar equation, often requiring coupled systems of equations instead. 
Extensive literature exists on the stabilization of coupled wave equations without delay \cite{Alabau2017,Alabau2013,Alabau2003,Alabau2002,Alabau-Leautaud2013,Avdonin,Bennour,Gerbi2021,Koumaiha2021,Mokhtari,LiuRao2009,Wehbe2011}; however, the impact of delays on such coupled systems remains less explored. For weakly coupled systems, some theoretical results have been achieved in \cite{oliveira2020stability,xu2020uniform,rebiai2014exponential}, while recent progress for strongly coupled systems can be found in \cite{ayechi2025indirect,akil2020stability,akil2021stability,silga2022indirect,moumni2024theoretical}.

Let $\Omega$ be a bounded open domain in $\mathbb{R}^n$ with a boundary $\partial \Omega$ of class $C^2$. In \cite{Gerbi2021}, the authors have investigated the energy decay of a system of two linear wave equations coupled through velocities described by (see also \cite{Alabau2017} for the case of a nonlinear damping):

\begin{equation}
	\label{problemwithoutdrlay}
	\begin{cases}
		u_{tt}(x,t) - \Delta u(x,t) + b(x) y_{t}(x,t) + a(x) u_t(x,t) = 0, & \text{in } \Omega \times (0, \infty), \\  
		y_{tt}(x,t) - \Delta y(x,t) - b(x) u_{t}(x,t) = 0, & \text{in } \Omega \times (0, \infty), \\  
		u(x,t) = y(x,t) = 0, & \text{on } \partial\Omega \times (0, \infty), \\  
		u(x,0) = u_0(x), \quad u_t(x,0) = u_1(x), & \text{in } \Omega, \\  
		y(x,0) = y_0(x), \quad y_t(x,0) = y_1(x), & \text{in } \Omega, \\  
		y_t(x,t) = g(x,t), & \text{in } \Omega \times (-\tau,0),
	\end{cases}
\end{equation}
where $a,b \in C^0(\bar{\Omega},\mathbb{R})$ and $a$ is a function satisfying
 $$
 a(x) \geq 0 \text { almost everywhere (a.e.) in } \Omega,
 $$
 and for a given open subset $\omega \subset \Omega$
 $$
 a(x)>a_0>0 \quad \text { a.e. in } \omega.
 $$
 Moreover, it is imposed that
 \begin{equation}
 	\label{conditionb}
 	\overline{\omega_b}\subset\left\{x\in\Omega:\quad b(x)\neq0\right\},
 \end{equation}
 where $\omega_b$ is a non-empty open subset of $\Omega$.
 
 In order to analyze the energy decay properties of system \eqref{problemwithoutdrlay}, the following assumptions have been considered:

 \begin{assumptions} \label{assumptionsgerbi}
 	The coupling region $\omega_b$ is contained within the damping region $\omega$ $(\omega_b \subset \omega)$, and $\omega_b$ satisfies the so-called Geometric Control Condition (GCC), introduced in \cite{Bardos} and recalled below in Definition \ref{GCC}.
 \end{assumptions}

\begin{assumptions} \label{assumptionsalabau}
	The coupling coefficient $b(x)$ is small and positive, satisfying $0 \leq b(x) \leq b_{+}$, $\forall x \in \Omega$, where $b_{+} \in (0, b^{*}]$ and $b^{*}$ is a constant depending on the domain and the control region. Additionally, there exists $b_{-} > 0$ such that $b(x) > b_{-}$ almost everywhere in $\omega_b$. Moreover, both the coupling region $\omega_b$ and the damping region $\omega$ satisfy an appropriate geometric condition known as the Piecewise Multipliers Geometric Condition (PMGC), introduced in \cite{liu1997locally}, applied in \cite{Alabau2017} and recalled below in Definition \ref{PMGC}.
\end{assumptions}
 
 Under Assumption \ref{assumptionsgerbi}, the authors in \cite{Gerbi2021} combined a frequency domain approach with the multiplier technique to establish the exponential decay of the corresponding energy for system \eqref{problemwithoutdrlay}. On the other hand, under Assumption \ref{assumptionsalabau}, the authors in \cite{Alabau2017} employed multiplier techniques to achieve similar exponential stability results. 
 
% These results highlight the crucial role of the coupling and damping regions, as well as the geometric conditions imposed, in ensuring the energy decay of the system. The interplay between velocity coupling and damping mechanisms provides a fundamental framework for understanding the long-term behavior of such coupled wave systems.
 
In this work, we aim to extend the results in \cite{Alabau2017,Gerbi2021} by considering a more complex setting where the system incorporates a time delay in one equation and nonlinear source terms. Specifically, we investigate the following coupled wave system:

\begin{equation}
	\label{coupledstabilityproblem}
	\begin{cases}
		u_{tt}(x,t) - \Delta u(x,t) + b(x) y_{t}(x,t) + a_1(x) u_t(x,t) = f_1(u(x,t)), & \text{in } \Omega \times (0, \infty), \\  
		y_{tt}(x,t) - \Delta y(x,t) - b(x) u_{t}(x,t) + a_2(x) y_t(x,t-\tau) = f_2(y(x,t)), & \text{in } \Omega \times (0, \infty), \\  
		u(x,t) = y(x,t) = 0, & \text{on } \partial\Omega \times (0, \infty), \\  
		u(x,0) = u_0(x), \quad u_t(x,0) = u_1(x), & \text{in } \Omega, \\  
		y(x,0) = y_0(x), \quad y_t(x,0) = y_1(x), & \text{in } \Omega, \\  
		y_t(x,t) = g(x,t), & \text{in } \Omega \times [-\tau,0].
	\end{cases}
\end{equation}
Here, the functions $a_1\in C^0(\bar{\Omega},\mathbb{R})$ and  $a_2 \in L^{\infty}(\Omega)$ are given coefficients. The coefficient $a_1$ satisfies the conditions:

$$
a_1(x) \geq 0 \quad \text{a.e in } \Omega,
$$

and

$$
a_1(x) > a_0 > 0 \quad \text{a.e in  } \omega \subset \Omega.
$$  

The initial state $(u_0, u_1, y_0, y_1, g)$ belongs to an appropriate function space, which will be specified later. Moreover, $f_1$ and $f_2$ are two nonlinear functions
satisfying suitable hypotheses that will be specified in the next sections. Finally, $\tau > 0$ represents the constant time delay.

Unlike previous studies that considered coupled wave systems where both the damping and delay terms appear in the same equation \cite{Ammari, silga2022indirect,oliveira2020stability,moumni2024theoretical,akil2020stability}, our focus is on the more challenging scenario where the damping and delay are distributed across different equations. This configuration complicates the stability analysis. Indeed,  in the classical setting where both mechanisms are in the same equation, the damping can directly counterbalance the delay effects (cf.\cite{nicaise2006stability}). In our case, this interaction is not straightforward, leading to additional mathematical challenges. This motivates our investigation into establishing new energy decay results for system \eqref{coupledstabilityproblem} under appropriate conditions.

Furthermore, we include nonlinear terms in the system, taking into account the presence of source terms.  Nonlinear effects are often present in delay differential equations that come from real-life applications. Various works have studied hyperbolic equations with suitable nonlinear source terms. For example, in \cite{ayechi2025indirect,Feng, FP2016,PP2022,P2024}, the authors study the stability of solutions for some abstract wave equations with nonlinear terms  satisfying some local Lipschitz continuity assumption. The presence of nonlinear terms makes the models more difficult to deal with and the stability issues require a more sophisticated analysis.

In practical situations, such as in elasticity and viscoelasticity, these nonlinearities naturally appear and lead to nonlinear source terms. In \cite{ACS2008,BM2006,LX2023}, the authors prove energy decay estimates for such systems.

To the best of our knowledge, this is the first work in which the destabilizing effect of a delay feedback present in one equation is compensated by means of a standard frictional damping in the other equation. Moreover, the source terms introduce stability issues too. This combination makes the mathematical analysis more delicate.

Our goal is, then, to explore how the interaction between delay, nonlinearity,  damping, and coupling mechanisms influences the stability properties of the system, and to derive sufficient conditions ensuring exponential energy decay.

Additionally, we also examine the particular case $\tau=0$, which corresponds to a system with indefinite damping:

\begin{equation}
	\label{indefiniteproblem} 
	\begin{cases}
		u_{tt}(x,t) - \Delta u(x,t) + b(x) y_{t}(x,t) + a_1(x) u_t(x,t) = f_1(u(x,t)), & \text{in } \Omega \times (0, \infty), \\  
		y_{tt}(x,t) - \Delta y(x,t) - b(x) u_{t}(x,t) + a_2(x) y_t(x,t) = f_2(y(x,t)), & \text{in } \Omega \times (0, \infty), \\   
		u(x,t) = y(x,t) = 0, & \text{on } \partial\Omega \times (0, \infty), \\  
		u(x,0) = u_0(x), \quad u_t(x,0) = u_1(x), & \text{in } \Omega, \\  
		y(x,0) = y_0(x), \quad y_t(x,0) = y_1(x), & \text{in } \Omega.
	\end{cases}
\end{equation}
This system represents a limiting case of our study but is considered separately since we can obtain the result under weaker conditions when $\tau=0$. In particular,  we establish exponential stability for \eqref{indefiniteproblem} under a suitable assumption on $a_2^{-}(x)=-\min\{a_2(x),0\}$. The study of indefinite damping functions in scalar wave equations has been considered in previous works, such as \cite{Haraux, freitas1996stability,menz2007exponential,benaddi2000energy,lopez1997linear} but for coupled wave  equations has not been done as far as we know.   
%%%%%%%%%%%%%%%%%%%%%%%%%%%%%%%%%%%%%%%%%%%%%%%%%%%%%%%%%

%In most physical settings, multiple waves can interact, overlap, and influence one another, creating a wide range of phenomena  \cite{nishikawa1968parametric}. This can happen in various contexts, including mechanical systems, fluid dynamics, acoustics, and electromagnetics. 
 
%%%%%%%%%%%%%%%%%%%%%%%%%%%%%%%%%%%%%%%%%%%%%%%%%%%%%%%%%%%%%%%%%%%%%%%%%%%%%
%%%%%%%%%%%%%%%%%%%%%%%%%%%%%%%%%%%%%%%%%%%%%%%%%%%%%%%%%%%%%%%%%%%%%%%%%%%%%%%%%%%%%%%%%%%%% Well-posedness %%%%%%%%%%%%%%%%%%%%%%%%%%%%%%%%%%%%%%%%%%%%%%%%%%%%%%%%%%%%%%%%%%%%%%%%%%%%%%%%%%%%%%%%%%%%%%%%%%%%%%%%%%%%%%%%%%%%%%%%%%%%%%%%%%%%%%%%%%%%%%%%%%%%%%

\section{The delayed coupled system} \label{section2}

In this section, we analyze the well-posedness and stability of the nonlinear delayed coupled system \eqref{coupledstabilityproblem}. 
 Before presenting our main result of this section, we recall   the GCC introduced by \cite{Bardos} and the PMGC introduced by 
 \cite{liu1997locally}.
\begin{definition}
	\label{GCC}
	We say that a subset $\omega$ of $\Omega$ satisfies the GCC if there exists a time $T>0$ such that every ray of the geometrical optics starting at any point $x \in \Omega$ at $t=0$ enters the region $\omega$ within time $T$.
\end{definition}
\begin{definition}
	\label{PMGC}
	 We say that a subset $\omega \subset \Omega$ satisfies the $P M G C$, if there exist subsets $\Omega_j \subset \Omega$ with Lipschitz boundaries and points $x_j \in \mathbb{R}^N, j=1, \ldots, L$, such that $\Omega_i \cap \Omega_j=\emptyset$ for $i \neq j$ and $\omega$ contains a neighborhood in $\Omega$ of the set $\cup_{j=1}^L \gamma_j\left(x_j\right) \cup\left(\Omega \backslash \cup_{j=1}^L \Omega_j\right)$, where $\gamma_j\left(x_j\right)=\left\{x \in \partial \Omega_j:\left(x-x_j\right) \cdot \nu_j(x) \geqslant 0\right\}$ and $\nu_j$ is the outward unit normal to $\partial \Omega_j$.
\end{definition}

 We introduce the energy space  
$$
\mathcal{H} = \left( H_0^1(\Omega) \times L^2(\Omega) \right)^2,
$$
equipped with the inner product  
$$
(U, \widetilde{U})_{\mathcal{H}} = \int_{\Omega} \nabla u_1 \cdot \nabla \widetilde{u}_1 \, dx + \int_{\Omega} u_2 \widetilde{u}_2 \, dx + \int_{\Omega} \nabla y_1 \cdot \nabla \widetilde{y}_1 \, dx + \int_{\Omega} y_2 \widetilde{y}_2 \, dx,
$$
for all $ U = (u_1, u_2, y_1, y_2) $ and $ \widetilde{U} = (\widetilde{u}_1, \widetilde{u}_2, \widetilde{y}_1, \widetilde{y}_2)$ in $ \mathcal{H}$. The associated norm is given by  
$$
\|U\|_{\mathcal{H}} = \sqrt{(U, U)_{\mathcal{H}}}.
$$

Setting $u_1 = u$, $u_2 = u_t$, $y_1 = y$, and $y_2 = y_t$, we define the state variable  
$$
U = \begin{bmatrix} u_1 \\ u_2 \\ y_1 \\ y_2 \end{bmatrix}, \quad U_0 = \begin{bmatrix} u_0 \\ u_1 \\ y_0 \\ y_1 \end{bmatrix}.
$$

We also introduce 
$$
\Psi(s) = \begin{bmatrix} 0 \\ 0 \\ 0 \\ a_2g(s) \end{bmatrix}, \quad
\mathcal{B} U(t) = \begin{bmatrix} 0 \\ 0 \\ 0 \\ a_2 y_2 \end{bmatrix}, \quad
\mathcal{F}(U(t)) = \begin{bmatrix} 0 \\ f_1(u(x,t)) \\ 0 \\ f_2(y(x,t)) \end{bmatrix}.$$

Next, we define the linear operator $ \mathcal{A}: D(\mathcal{A}) \subset \mathcal{H} \to \mathcal{H}$ by  
$$
D(\mathcal{A}) = \left( (H^2(\Omega) \cap H_0^1(\Omega)) \times H_0^1(\Omega) \right)^2, \quad 
\mathcal{A} U = \begin{bmatrix} u_2 \\ \Delta u_1 - b y_2 - a_1 u_2 \\ y_2 \\ \Delta y_1 + b u_2 \end{bmatrix}.$$

It is well known that the operator $\mathcal{A}$ is non-negative and generates an exponentially stable contraction semigroup $(S(t))_{t \geq 0}$ on $\mathcal{H}$ (see \cite[Remark 3.3]{Alabau2017} and \cite[Theorem 3.4]{Gerbi2021}) under either Assumption \ref{assumptionsgerbi}  or Assumption \ref{assumptionsalabau}. More precisely, there exist constants $M > 0$ and $\alpha > 0$ such that  

\begin{equation} \label{exp_decay}
	\|S(t)\|_{\mathcal{L}(\mathcal{H})} \leq M e^{- \alpha t}, \quad \forall t \geq 0.
\end{equation}

%This ensures that the energy of the system decays exponentially over time, provided that one of these assumptions holds.

Thus, the delayed coupled system \eqref{coupledstabilityproblem} can be rewritten in the following abstract form:  

\begin{equation}
	\label{cauchyproblem}
\begin{cases}
	\dot{U}(t) = \mathcal{A} U(t) - \mathcal{B} U(t-\tau) + \mathcal{F}(U(t)), & t > 0, \\
	U(0) = U_0, & \\
\mathcal{B}	U(s) = \Psi(s), & s \in [-\tau, 0].
\end{cases}
\end{equation}
Observe that, if  $U_0 \in \mathcal{H}$, then the following Duhamel formula holds: 
	\begin{equation}
	U(t)=S(t) U_0+\int_0^t S(t-s)  \mathcal{B} U(s-\tau) d s+\int_0^t S(t-s) \mathcal{F}(U(s)) d s.
	\end{equation}
We begin by establishing the local existence and uniqueness of solutions, as stated in the following lemma.
	\begin{lemma}
		\label{lemmalocalexistence}
		Let us consider the system \eqref{cauchyproblem} with initial data $U_0\in\mathcal{H}$ and $\Psi\in 
		C([-\tau, 0]; \mathcal{H}).$ Then, there exists a unique continuous local solution $U$ defined on
		a time interval $[0,\delta)$, for some $\delta>0$.
	\end{lemma}
	\begin{proof} Note that for $t\in (0, \tau),$ $t-\tau\in (-\tau, 0).$ Then, for $t\in (0, \tau),$ we have $\mathcal{B} U(t-\tau)=\Psi(t-\tau)=a_2g(t-\tau)$ and 
		we can rewrite the abstract system \eqref{cauchyproblem} as an non-delayed problem:
		\begin{equation*}
			\begin{cases}
				\dot{U}(t) = \mathcal{A} U(t)  -a_2g(t-\tau) + \mathcal{F}(U(t)), & t\in(0,\tau), \\
				U(0) = U_0. & 
			\end{cases}
		\end{equation*}
		Then, we can apply the classical theory of nonlinear semigroups \cite{pazy2012semigroups}
		obtaining the existence of a unique solution on a set $[0,\delta)$, with $\delta\leq\tau$.
	\end{proof}
 In order to treat the global existence and uniqueness as well as the exponential stability of the system \eqref{cauchyproblem}, we make the following assumptions on $f_1$ and $f_2$:
 \begin{assumptions}
 	\label{nonlinearassumption}
$f_1,f_2$ are continuous functions such that
  \begin{enumerate}
  	\item $f_1(0)=f_2(0)=0$;
 \item  For all $r>0$ there exists a constant $L(r)>0$ such that, for all $u, v \in  H^{1}_{0}(\Omega)$ satisfying $\left\|\nabla u\right\|_{L^2(\Omega)} \leq r$ and $\left\|\nabla v\right\|_{L^2(\Omega)} \leq r$, one has
 $$
 \|f_i(u)-f_i(v)\|_{L^2(\Omega)} \leq L(r)\left\|\nabla u-\nabla v\right\|_{L^2(\Omega)}
 $$
 for $i=1,2$; and $L(r)\rightarrow 0$ as $r\rightarrow 0.$
 \item  There exist two  strictly increasing continuous functions $h_1,h_2: \mathbb{R}_{+} \rightarrow \mathbb{R}_{+},$ $h_1(0)=h_2(0)=0,$ such that 
 \begin{equation}
 	\label{hypothesisnonlinearfunc}
\left\vert \int_{\Omega} f_i(u) u\,dx \right\vert\leq h_i\left(\left\|\nabla u\right\|_{L^2(\Omega)}\right)\left\|\nabla u\right\|_{L^2(\Omega)}^2
  \end{equation}
 for all $u \in H^{1}_{0}(\Omega)$ and $i=1,2$.
   \end{enumerate}
 \end{assumptions}
 Furthermore, thanks to Hypothesis \ref{nonlinearassumption}, we have that  $\mathcal{F}(0)=0$ and for any $r>0$ there exists a constant $L(r)>0$ such that
 
 $$
 \|\mathcal{F}(U)-\mathcal{F}(V)\|_{\mathcal{H}} \leq L(r)\|U-V\|_{\mathcal{H}}
 $$
 
 whenever $\|U\|_{\mathcal{H}} \leq r$ and $\|V\|_{\mathcal{H}} \leq r$. In particular,
 
 $$
 \|\mathcal{F}(U)\|_{\mathcal{H}} \leq L(r)\|U\|_{\mathcal{H}}.
 $$
Let us define
 \begin{equation}
 	\label{definitionnonlinearfunc}
 	F_i(s):=\int_{0}^{s}f_i(r)\,dr,\quad \forall  s\in \mathbb{R},
 \end{equation} 
 for $i=1,2$.
 
Under Hypothesis \ref{nonlinearassumption}, we derive the following estimates for the nonlinear terms $\int_{\Omega}F_1(u(x))\,dx$ and $\int_{\Omega}F_2(u(x))\,dx$.  

\begin{lemma}
	\label{lemmanonlinear}
	Assume Hypothesis \ref{nonlinearassumption} holds. Then, for all $u\in H^1_0(\Omega)$ and $i=1,2$, we have  
	\begin{equation}
	\label{estimatenonlinearterm}
\left\vert\int_{\Omega}F_i(u(x))\,dx\right\vert\leq \frac{1}{2}h_i(\|\nabla u\|_{L^2(\Omega)})\|\nabla u\|_{L^2(\Omega)}^2.
\end{equation}
\end{lemma}

\begin{proof}
	Let $u\in H^1_0(\Omega)$ and $i\in\{1,2\}$. Using the identity  
	$$  
	\frac{d}{ds} F_i(su) = F_i^{\prime}(su)u = f_i(su)u,  
	$$  
	we obtain  
	$$  
	\int_{\Omega}F_i(u(x))\,dx = \int_{\Omega} \int_{0}^{1} f_i(su)u\,ds\,dx = \int_{0}^{1} \int_{\Omega} f_i(su)su\,dx\,\frac{ds}{s}.  
	$$  
	Applying \eqref{hypothesisnonlinearfunc}, we deduce  
	$$  
	\begin{aligned}
		\left\vert\int_{\Omega}F_i(u(x))\,dx \right\vert &\leq \int_{0}^{1} h_i(\|\nabla(su)\|_{L^2(\Omega)})\|\nabla(su)\|_{L^2(\Omega)}^2\,\frac{ds}{s} \\
		&\leq \int_{0}^{1} h_i(\|\nabla(su)\|_{L^2(\Omega)})s^2\|\nabla u\|_{L^2(\Omega)}^2\,\frac{ds}{s} \\
		&\leq \frac{1}{2}h_i(\|\nabla u\|_{L^2(\Omega)})\|\nabla u\|_{L^2(\Omega)}^2.
	\end{aligned}
	$$  
\end{proof}
In the spirit of \cite{PP2022}, in order to have the existence and uniqueness of a global solution of the system \eqref{cauchyproblem},
we introduce the following energy functional associated to system \eqref{coupledstabilityproblem}:
\begin{definition}
	Let $(u,y)$  be a mild solution of \eqref{coupledstabilityproblem} and define its energy by
	\begin{equation}
		\begin{aligned}
			E(t)&=\frac{1}{2}\int_{\Omega}\left\{u_t^2(x,t)+|\nabla u(x,t)|^2+y_t^2(x,t)+|\nabla y(x,t)|^2\right\}\,dx\\
			&-\int_{\Omega}F_1(u(x,t))\,dx-\int_{\Omega}F_2(y(x,t))\,dx+\frac{1}{2}\int_{\Omega}\int_{t-\tau}^{t}|a_2(x)|y_t^2(x,s)\,ds\,dx.
		\end{aligned}
	\end{equation}
\end{definition}
We can deduce a first estimate for small enough initial data.
\begin{lemma}
	\label{lemmaminenergy1}
	Assume Hypothesis \ref{nonlinearassumption}.  Let $U=(u, u_t, y, y_t)$ be a non-trivial solution of  \eqref{cauchyproblem} defined on an interval $[0, T),$ for some $T>0.$
 If $h_1(\|\nabla u_0\|_{L^2(\Omega)})<\frac{1}{2}$ and  $h_2(\|\nabla y_0\|_{L^2(\Omega)})<\frac{1}{2}$, then $E(0)>0$.
\end{lemma}
\begin{proof}
	From Lemma \ref{lemmanonlinear}, we can infer that
	$$\int_{\Omega}F_1(u_0(x))\,dx\leq \frac{1}{2}h_1(\|\nabla u_0\|_{L^2(\Omega)})\|\nabla u_0\|_{L^2(\Omega)}^2\leq\frac{1}{4}\|\nabla u_0\|_{L^2(\Omega)}^2$$
	and 
	$$\int_{\Omega}F_2(y_0(x))\,dx\leq \frac{1}{2}h_2(\|\nabla y_0\|_{L^2(\Omega)})\|\nabla y_0\|_{L^2(\Omega)}^2\leq\frac{1}{4}\|\nabla y_0\|_{L^2(\Omega)}^2.$$
	As a consequence,
	$$	\begin{aligned}
		E(0)&=\frac{1}{2}\|u_1\|_{L^2(\Omega)}^2+\frac{1}{2}\|\nabla u_0\|_{L^2(\Omega)}^2+\frac{1}{2}\|y_1\|_{L^2(\Omega)}^2+\frac{1}{2}\|\nabla y_0\|_{L^2(\Omega)}^2\\
		&-\int_{\Omega}F_1(u_0(x))\,dx-\int_{\Omega}F_2(y_0(x))\,dx+\frac{1}{2}\int_{-\tau}^{0}\|\sqrt{|a_2|}y_t(s)\|_{L^2(\Omega)}^2\,ds\,dx\\
		&>\frac{1}{4}\|u_1\|_{L^2(\Omega)}^2+\frac{1}{4}\|\nabla u_0\|_{L^2(\Omega)}^2+\frac{1}{4}\|y_1\|_{L^2(\Omega)}^2+\frac{1}{4}\|\nabla y_0\|_{L^2(\Omega)}^2\\
		&+\frac{1}{4}\int_{-\tau}^{0}\|\sqrt{|a_2|}y_t(s)\|_{L^2(\Omega)}^2\,ds\,dx.
	\end{aligned}$$
	Therefore, since $U$ is a non-zero solution, we have proven that $E(0)>0$.
\end{proof}
We can also obtain the following preliminary estimate.
\begin{proposition}
	\label{Propmajenergy}
	Let $(u,y)$ be a solution to \eqref{coupledstabilityproblem} defined on an interval $[0, T),$ for some $T>0.$ If $E(t)\geq \frac{1}{4}\|y_t(t)\|^2_{L^2(\Omega)}$, for any $t\in [0, T)$, then 
	\begin{equation}
		\label{estimateenergy}
		E(t)\leq C(t)E(0),
	\end{equation}
	for any $t\geq 0$, where 
	\begin{equation}
		\label{functionC}
		C(t)=e^{4\|a_2\|_{L^{\infty}(\Omega)}t}.
	\end{equation}
\end{proposition}
\begin{proof}
	Differentiating formally $E$ with respect to $t$, we
	obtain
	$$\begin{aligned}
		E^{\prime}(t)&=\int_{\Omega}\left\{u_t(x,t) u_{tt}(x,t)+\nabla u(x,t)\nabla u_t(x,t) +y_t(x,t)y_{tt}(x,t)+\nabla y(x,t)\nabla y_t(x,t)\right\}\,dx\\
		&-\int_{\Omega}u_t(x,t)f_1(u(x,t))\,dx-\int_{\Omega}y_t(x,t)f_2(y(x,t))\,dx+\frac{1}{2}\int_{\Omega}|a_2(x)|y_t^2(x,t)\,dx\\
		&-\frac{1}{2}\int_{\Omega}|a_2(x)|y_t^2(x,t-\tau)\,dx.
	\end{aligned}$$
	By using the equations satisfied by $u$ and $y$ in  \eqref{coupledstabilityproblem} and Green formula, we get
	$$\begin{aligned}
		E^{\prime}(t)&=\int_{\Omega}\left\{u_t(x,t) \left(\Delta u(x,t)-by_t(x,t)-a_1(x)u_t(x,t)+f_1(u(x,t))\right)+\nabla u(x,t)\nabla u_t(x,t)\right.\\
		&\left.+y_t(x,t)\left(\Delta y(x,t)+b u_{t}(x,t)-a_2(x)y_t(x,t-\tau)+f_2(y(x,t))\right)+\nabla y(x,t)\nabla y_t(x,t)\right\}\,dx\\
		&-\int_{\Omega}u_t(x,t)f_1(u(x,t))\,dx-\int_{\Omega}y_t(x,t)f_2(y(x,t))\,dx+\frac{1}{2}\int_{\Omega}|a_2(x)|y_t^2(x,t)\,dx\\
		&-\frac{1}{2}\int_{\Omega}|a_2(x)|y_t^2(x,t-\tau)\,dx\\
		=&-\int_{\Omega}a_1(x)u_t^2(x,t)\,dx-
		\int_{\Omega}a_2(x)y_t(x,t)y_t(x,t-\tau)\,dx+\frac{1}{2}\int_{\Omega}|a_2(x)|y_t^2(x,t)\,dx\\
		&-\frac{1}{2}\int_{\Omega}|a_2(x)|y_t^2(x,t-\tau)\,dx\\
		&\leq\int_{\Omega}|a_2(x)|y_t^2(x,t)\,dx\leq \|a_2\|_{L^{\infty}(\Omega)}\|y_t(t)\|_{L^2(\Omega)}^2.
	\end{aligned}$$
	Since we have assumed $E(t)\geq \frac{1}{4}\|y_t(t)\|_{L^2(\Omega)}^2$
	for any $t\in [0, T)$, we obtain
	$$E^{\prime}(t)\leq 4\|a_2\|_{L^{\infty}(\Omega)}E(t).$$
	By Gronwall's inequality, we then obtain \eqref{estimateenergy}.
\end{proof}
 Next, we can give the following estimate from below for sufficiently small initial data.
\begin{lemma}
	\label{lemmaminenergy}
	Assume Hypothesis \ref{nonlinearassumption}. Let $U$ be a non-trivial solution of  \eqref{cauchyproblem} defined on an interval $[0,\delta)$, and fix a time $T\geq\delta$.  Then, 
   if $h_1(\|\nabla u_0\|_{L^2(\Omega)})<\frac{1}{2}$,  $h_2(\|\nabla y_0\|_{L^2(\Omega)})<\frac{1}{2}$ and $h_i(2\sqrt{C(T)E(0)})<\frac{1}{2}$ for $i=1,2$,  with $C(\cdot)$ defined as in \eqref{functionC},  then 
		\begin{equation}
			\label{minorationenergy}
			\begin{aligned}
				E(t)&>\frac{1}{4}\|u_t(t)\|^2_{L^2(\Omega)}+\frac{1}{4}\|\nabla u(t)\|^2_{L^2(\Omega)}+\frac{1}{4}\|y_t(t)\|^2_{L^2(\Omega)}+\frac{1}{4}\|\nabla y(t)\|^2_{L^2(\Omega)}\\
				&+\frac{1}{4}\int_{t-\tau}^{t}\|\sqrt{|a_2|}y_t(s)\|^2_{L^2(\Omega)}\,ds
			\end{aligned}
		\end{equation}
		for all $t\in[0,\delta)$. In particular,
		$$E(t)>\frac{1}{4}\|U(t))\|^2_{\mathcal{H}},\quad  \forall t\in[0,\delta).$$
\end{lemma}
\begin{proof}
	We argue by contradiction. Let us denote
	$$r:=\sup\{s\in [0,\delta)\text{ : \eqref{minorationenergy} holds } \forall t\in[0,s)\}.$$
	We suppose by contradiction that $r<\delta$. Then, by continuity, we have
	$$	\begin{aligned}
		E(r)&=\frac{1}{4}\|u_t(r)\|^2_{L^2(\Omega)}+\frac{1}{4}\|\nabla u(r)\|^2_{L^2(\Omega)}+\frac{1}{4}\|y_t(r)\|^2_{L^2(\Omega)}+\frac{1}{4}\|\nabla y(r)\|^2_{L^2(\Omega)}\\
		&+\frac{1}{4}\int_{r-\tau}^{r}\|\sqrt{|a_2|}y_t(s)\|^2_{L^2(\Omega)}\,ds.
	\end{aligned}$$
	Then, we can see that 
	$$E(r)\geq \frac{1}{4}\|y_t(r)\|^2_{L^2(\Omega)},\quad E(r)\geq \frac{1}{4}\|\nabla u(r)\|^2_{L^2(\Omega)}\quad\text{and}\quad E(r)\geq \frac{1}{4}\|\nabla y(r)\|^2_{L^2(\Omega)}.$$ 
	Using the fact that $h_i$ for $i=1,2$ are increasing functions and Proposition \ref{Propmajenergy}, we get
	$$h_1(\|\nabla u(r)\|_{L^2(\Omega)})\leq h_1(2\sqrt{E(r)})\leq h_1(2\sqrt{C(r)E(0)})<\frac{1}{2}$$
	and 
	$$h_2(\|\nabla y(r)\|_{L^2(\Omega)})\leq h_2(2\sqrt{E(r)})\leq h_2(2\sqrt{C(r)E(0)})<\frac{1}{2}.$$
	Hence using the definition of the energy and \eqref{estimatenonlinearterm}, we conclude that
	$$	\begin{aligned}
		E(r)&=\frac{1}{2}\|u_t(r)\|_{L^2(\Omega)}^2+\frac{1}{2}\|\nabla u(r)\|_{L^2(\Omega)}^2+\frac{1}{2}\|y_t(r)\|_{L^2(\Omega)}^2+\frac{1}{2}\|\nabla y(r)\|_{L^2(\Omega)}^2\\
		&-\int_{\Omega}F_1(u(x,r))\,dx-\int_{\Omega}F_2(y(x,r))\,dx+\frac{1}{2}\int_{r-\tau}^{r}\|\sqrt{|a_2|}y_t(s)\|_{L^2(\Omega)}^2\,ds\,dx\\
		&>\frac{1}{4}\|u_t(r)\|_{L^2(\Omega)}^2+\frac{1}{4}\|\nabla u(r)\|_{L^2(\Omega)}^2+\frac{1}{4}\|y_t(r)\|_{L^2(\Omega)}^2+\frac{1}{4}\|\nabla y(r)\|_{L^2(\Omega)}^2\\
		&+\frac{1}{4}\int_{r-\tau}^{r}\|\sqrt{|a_2|}y_t(s)\|_{L^2(\Omega)}^2\,ds\,dx.
	\end{aligned}$$
	This contradicts the maximality 
	of $r$. Hence, $r=\delta$ and this concludes the proof of the lemma.
\end{proof}
Next, we establish an exponential decay result for the abstract model \eqref{cauchyproblem} under an appropriate  assumption on the function $a_2$. Let us consider the following hypothesis
\begin{assumptions}
	\label{assumptionona2}
	The function $a_2\in L^{\infty}(\Omega)$ satisfies 
	$$e^{\alpha\tau}\,\|a_2\|_{L^{\infty}(\Omega)}<\frac{\alpha}{M},$$
with $M$ and $\alpha$ as in \eqref{exp_decay}.
\end{assumptions}
\begin{theorem}
	\label{Theoremexpstability}
 Assume that either Assumption \ref{assumptionsgerbi} or Assumption \ref{assumptionsalabau} holds, together with Hypotheses \ref{nonlinearassumption} and \ref{assumptionona2}.  	Let  $U \in C([0,T) ; \mathcal{H})$ a solution of the system \eqref{cauchyproblem}  satisfying $\|U(t)\|_{\mathcal{H}} \leq C_\rho$ for all $t\in[0,T)$ with $C_{\rho}$ such that  $L\left(C_\rho\right)<\frac{\alpha-M\|a_2\|_{L^{\infty}(\Omega)}e^{\alpha\tau}}{M}$.
Then, the  following exponential decay estimate holds
	\begin{equation}
		\label{ineqexpstability}
	\|U(t)\|_{\mathcal{H}} \leq M\left(\left\|U_0\right\|_{\mathcal{H}}+\int_0^\tau e^{\alpha s}\|\Psi(s-\tau)\|_{\mathcal{H}} d s\right) e^{-\left(\alpha-M\|a_2\|_{L^{\infty}(\Omega)}e^{\alpha\tau}-M L\left(C_\rho\right)\right) t},
		\end{equation}
	for any $t\in[0,T)$.
\end{theorem}
\begin{proof}
	Applying Duhamel's formula, we obtain
	\begin{equation}
		\label{ineqDuhamel}
		\begin{aligned}
			\|U(t)\|_{\mathcal{H}} \leq & M e^{-\alpha t}\left\|U_0\right\|_{\mathcal{H}}+M e^{-\alpha t} \int_0^t e^{ \alpha s} \|\mathcal{B} U(s-\tau)\|_{\mathcal{H}} d s \\
			& +M L\left(C_\rho\right) e^{- \alpha t} \int_0^t e^{ \alpha s}\|U(s)\|_{\mathcal{H}} d s,
		\end{aligned}
	\end{equation}
	where we have used the assumption $\|F(U(t))\|_{\mathcal{H}} \leq L(C_\rho)\|U(t)\|_{\mathcal{H}}$ for all $t \geq 0$.
	
	We distinguish the following two cases:
	
	\begin{itemize}
		\item \textbf{Case $\tau \leq t$} : Using \eqref{ineqDuhamel}, we obtain
		\begin{equation*}
			\begin{aligned}
				& \|U(t)\|_{\mathcal{H}} \leq M e^{- \alpha t}\left\|U_0\right\|_{\mathcal{H}}+M e^{- \alpha t} \int_0^\tau e^{ \alpha s}\|\Psi(s-\tau)\|_{\mathcal{H}} d s \\
				& \quad+M e^{- \alpha t} \int_\tau^t e^{ \alpha s} \|a_2\|_{L^{\infty}(\Omega)} \|U(s-\tau)\|_{\mathcal{H}} d s+M L(C_\rho) e^{- \alpha t} \int_0^t e^{ \alpha s}\|U(s)\|_{\mathcal{H}} d s.
			\end{aligned}
		\end{equation*}
		By applying the change of variables $s-\tau=z$  in the second integral, it follows that
		\begin{equation*}
			\begin{aligned}
				& e^{ \alpha t}\|U(t)\|_{\mathcal{H}} \leq M\| U_0\|_{\mathcal{H}}+M\int_0^\tau e^{ \alpha s} \| \Psi(s-\tau) \|_{\mathcal{H}} d s \\
				&+M \|a_2\|_{L^{\infty}(\Omega)} e^{ \alpha \tau} \int_0^t e^{ \alpha z}\|U(z)\|_{\mathcal{H}} d z+M L(C_\rho) \int_0^t e^{ \alpha s}\|U(s)\|_{\mathcal{H}} d s.
			\end{aligned}
		\end{equation*}
		Defining $\tilde{u}(t):=e^{ \alpha t}\|U(t)\|_{\mathcal{H}}$ and $M_0:=M\| U_0\|_{\mathcal{H}}+M\int_0^\tau e^{ \alpha s} \|\Psi(s-\tau) \|_{\mathcal{H}} d s$, we get
		\begin{equation*}
			\tilde{u}(t) \leq M_0+\int_0^t\left( M \|a_2\|_{L^{\infty}(\Omega)} e^{ \alpha \tau}+M L(C_\rho) \right) \tilde{u}(s)d s.
		\end{equation*}
		By Gronwall's inequality, it follows that
		\begin{equation*}
			\tilde{u}(t) \leq M_0e^{( M \|a_2\|_{L^{\infty}(\Omega)} e^{ \alpha \tau}+M L(C_\rho) )t}.
		\end{equation*}
		Using the definition of $\tilde{u}$ and assumption \ref{assumptionona2}, we obtain \eqref{ineqexpstability}.
		
		\item \textbf{Case $0\le t< \tau$} : From \eqref{ineqDuhamel}, it follows that
		\begin{equation*}
			\begin{aligned}
				& \|U(t)\|_{\mathcal{H}} \leq M e^{- \alpha t}\left\|U_0\right\|_{\mathcal{H}}+M e^{- \alpha t} \int_0^\tau e^{ \alpha s}\|\Psi(s-\tau)\|_{\mathcal{H}} d s \\
				& \quad+M L(C_\rho) e^{- \alpha t} \int_0^t e^{ \alpha s}\|U(s)\|_{\mathcal{H}} d s.
			\end{aligned}
		\end{equation*}
		Thus,
		\begin{equation*}
			\tilde{u}(t) \leq M_0+\int_0^t M L(C_\rho) \tilde{u}(s)d s.
		\end{equation*}
		By Gronwall's inequality, we obtain
		\begin{equation*}
			\tilde{u}(t) \leq M_0e^{M L(C_\rho) t}.
		\end{equation*}
		As before, using assumption \ref{assumptionona2}, we obtain \eqref{ineqexpstability}.
	\end{itemize}
\end{proof}
We are now ready to prove the global existence and uniqueness for  solutions of   system \eqref{coupledstabilityproblem}.
\begin{theorem}
 Assume that either Assumption \ref{assumptionsgerbi} or Assumption \ref{assumptionsalabau} holds, together with Hypotheses \ref{nonlinearassumption} and \ref{assumptionona2}.  We consider the  initial data $U_0\in\mathcal{H}$  and $\Psi\in C([-\tau,0];\mathcal{H})$ such that 
 \begin{equation}
 	\label{smallnessinitialdata}
 	\left\|U_0\right\|_{\mathcal{H}}^2+\int_0^\tau \|\sqrt{|a_2|}g(s-\tau)\|_{L^2(\Omega)}^2 d s <\rho^2,
 \end{equation}
 with $\rho$ sufficiently small constant.
 Then, the problem \eqref{cauchyproblem} has a unique global solution exponentially decaying according to  \eqref{ineqexpstability}.
\end{theorem}
\begin{proof}
	Let us fix a time $T>\tau$  such that
	$$C_T:=2M^2(1+\tau e^{2 \alpha\tau}\Vert a_2\Vert_\infty)(1+\tau\|a_2\|_{L^{\infty}(\Omega)}e^{ \alpha\tau})e^{-\left( \alpha-M\|a_2\|_{L^{\infty}(\Omega)}e^{ \alpha\tau}\right) T}< 1.$$	
	Moreover let $\rho>0$ be such that
	$$\rho\leq \frac{1}{2\sqrt{C(T)}}\min_{i=1,2}\left\{h_i^{-1}\left(\frac{1}{2}\right)\right\}.$$
	We consider initial data such that
$$\|u_1\|_{L^2(\Omega)}^2+\|\nabla u_0\|_{L^2(\Omega)}^2+\|y_1\|_{L^2(\Omega)}^2+\|\nabla y_0\|_{L^2(\Omega)}^2
+\int_{-\tau}^{0}\|\sqrt{|a_2|}g(s)\|_{L^2(\Omega)}^2\,ds\,<\rho^2.$$
We observe that this is equivalent to 
\begin{equation}
	\label{assymptiononinitialdata}
\|U_0\|_{\mathcal{H}}^2
+\int_{0}^{\tau}\|\sqrt{|a_2|}g(s-\tau)\|_{L^2(\Omega)}^2\,ds\, < \rho^2.
\end{equation}
Now, by  Lemma \ref{lemmalocalexistence} we know that there exists a local solution $(u,y)$  of  \eqref{coupledstabilityproblem} on a time interval $[0,\delta)$. Without  loss of generality, 
we can assume $\delta<T$ (eventually, we can take a larger $T$). From our assumption on the initial data, on the monotonicity of $h_i$ for $i=1,2$ and
the fact that $C(T)>1$, we have 
$$h_1(\|\nabla u_0\|_{L^2(\Omega)})\leq h_1(\rho)\leq h_1\left(\frac{1}{2\sqrt{C(T)}}h_1^{-1}\left(\frac{1}{2}\right)\right)<\frac{1}{2}$$
and 
$$h_2(\|\nabla y_0\|_{L^2(\Omega)})\leq h_2(\rho)\leq h_2\left(\frac{1}{2\sqrt{C(T)}}h_2^{-1}\left(\frac{1}{2}\right)\right)<\frac{1}{2}.$$
Hence, by Lemma \ref{lemmaminenergy1}, $E(0)>0$. Moreover, from \eqref{estimatenonlinearterm}, we obtain
\begin{equation}
	\label{ineqhelp0}
\begin{aligned}
E(0)&\leq \frac{1}{2}\|u_1\|_{L^2(\Omega)}^2+\frac{3}{4}\|\nabla u_0\|_{L^2(\Omega)}^2+\frac{1}{2}\|y_1\|_{L^2(\Omega)}^2\\
{}&\hspace{0.5 cm}+\frac{3}{4}\|\nabla y_0\|_{L^2(\Omega)}^2+\frac{1}{2}\int_{-\tau}^{0}\|\sqrt{|a_2|}g(s)\|_{L^2(\Omega)}^2\,ds\,dx\leq \rho^2,
\end{aligned}
\end{equation}
which gives us
$$h_i(2\sqrt{C(T)E(0)})< h_i(2\sqrt{C(T)}\rho)<h_i\left(h_i^{-1}\left(\frac{1}{2}\right)\right)=\frac{1}{2}$$
for $i=1,2$. Hence, we can apply Lemma \ref{lemmaminenergy} to obtain
\begin{equation}
	\label{ineqhelp1}
\begin{aligned}
	E(t)&>\frac{1}{4}\|u_t(t)\|^2_{L^2(\Omega)}+\frac{1}{4}\|\nabla u(t)\|^2_{L^2(\Omega)}+\frac{1}{4}\|y_t(t)\|^2_{L^2(\Omega)}+\frac{1}{4}\|\nabla y(t)\|^2_{L^2(\Omega)}\\
	&+\frac{1}{4}\int_{t-\tau}^{t}\|\sqrt{|a_2|}y_t(s)\|^2_{L^2(\Omega)}\,ds>0
\end{aligned}
\end{equation}
for all $t\in[0,\delta)$; in particular $E(t)\geq \frac{1}{4}\|y_t(t)\|^2_{L^2(\Omega)}$ for all $t\in[0,\delta)$. Thus, applying Proposition \ref{Propmajenergy} and the fact that $\delta<T$ we have 
\begin{equation}
	\label{ineqhelp2}
	E(t)\leq C(t)E(0),\quad \forall t\in[0,\delta).
\end{equation}
As a consequence, from \eqref{ineqhelp1} and \eqref{ineqhelp2}, we obtain
$$\begin{aligned}
\frac{1}{4}\|\nabla u(t)\|^2_{L^2(\Omega)}&\leq \frac{1}{4}\|u_t(t)\|^2_{L^2(\Omega)}+\frac{1}{4}\|\nabla u(t)\|^2_{L^2(\Omega)}+\frac{1}{4}\|y_t(t)\|^2_{L^2(\Omega)}+\frac{1}{4}\|\nabla y(t)\|^2_{L^2(\Omega)}\\
&+\frac{1}{4}\int_{t-\tau}^{t}\|\sqrt{|a_2|}y_t(s)\|^2_{L^2(\Omega)}\,ds<E(t)\leq C(T)E(0),
\end{aligned}$$
and 
$$\begin{aligned}
	\frac{1}{4}\|\nabla y(t)\|^2_{L^2(\Omega)}&\leq \frac{1}{4}\|u_t(t)\|^2_{L^2(\Omega)}+\frac{1}{4}\|\nabla u(t)\|^2_{L^2(\Omega)}+\frac{1}{4}\|y_t(t)\|^2_{L^2(\Omega)}+\frac{1}{4}\|\nabla y(t)\|^2_{L^2(\Omega)}\\
	&+\frac{1}{4}\int_{t-\tau}^{t}\|\sqrt{|a_2|}y_t(s)\|^2_{L^2(\Omega)}\,ds<E(t)\leq C(T)E(0),
\end{aligned}$$
for all $t\in[0,\delta)$. Then, we can extend the solution in $t=\delta$ and on the whole interval $[0,T]$. In particular, for $t=T$, we get
$$h_1(\|\nabla u(T)\|_{L^2(\Omega)})\le h_1(2\sqrt{C(T)E(0)})< \frac{1}{2}$$
and 
$$h_2(\|\nabla y(T)\|_{L^2(\Omega)})\le h_2(2\sqrt{C(T)E(0)})< \frac{1}{2}.$$
%Then \eqref{ineqhelp1} holds for all $t\in[0,T]$.
By \eqref{ineqhelp1} and \eqref{ineqhelp0} we deduce 
$$\frac{1}{4}\|U(t)\|_{\mathcal{H}}^2\leq E(t)\leq C(T)E(0)\leq C(T)\rho^2, \quad\forall t\in[0,T],$$
i.e 
$$\|U(t)\|_{\mathcal{H}}\leq C_\rho, \quad\forall t\in[0,T],$$
where $C_\rho:=2 \sqrt{C(T)}\rho$.
Now, eventually choosing smaller values of $\rho$, we suppose that $\rho$
is such that $L\left(C_\rho\right)<\frac{\alpha-M\|a_2\|_{L^{\infty}(\Omega)}e^{ \alpha\tau}}{2M}$.  %Consequently, Hypothesis $(I)$ is satisfied in the interval $[0,T]$.
Hence, Theorem \ref{Theoremexpstability}  gives us the following estimate:
	\begin{equation}
		\label{majuint}
		\|U(t)\|_{\mathcal{H}} \leq M\left(\left\|U_0\right\|_{\mathcal{H}}+\int_0^\tau e^{ \alpha s}\|\Psi(s-\tau)\|_{\mathcal{H}} d s\right) e^{-\frac{\left( \alpha-M\|a_2\|_{L^{\infty}(\Omega)}e^{ \alpha\tau}\right)}{2} t},
	\end{equation}
	 for any $t\in[0,T]$. By Hypothesis \eqref{assymptiononinitialdata}  and the Hölder inequality, it follows
	 $$\int_0^\tau e^{ \alpha s}\|\Psi(s-\tau)\|_{\mathcal{H}} d s=\int_{0}^{\tau}e^{ \alpha s}\|a_2g(s-\tau)\|_{L^2(\Omega)}\,ds\leq e^{ \alpha \tau}\sqrt{\tau}\rho \Vert a_2\Vert_\infty^{1/2}.$$
	 Hence, from \eqref{majuint}, we obtain
	 $$	\|U(t)\|_{\mathcal{H}}^2 \leq 2M^2\rho^2(1+\tau e^{2 \alpha\tau}  \Vert a_2\Vert_\infty)e^{-\left( \alpha-M\|a_2\|_{L^{\infty}(\Omega)}e^{ \alpha\tau}\right) t}.$$
	 for any $t\in[0,T]$. Moreover 
	 $$\int_{T-\tau}^{T} \|\sqrt{|a_2|}y_t(s)\|_{L^2(\Omega)}^2\, ds\leq 2M^2\tau\rho^2\|a_2\|_{L^{\infty}(\Omega)}e^{ \alpha\tau}(1+\tau e^{2 \alpha\tau} \Vert a_2\Vert_\infty)e^{-\left( \alpha-M\|a_2\|_{L^{\infty}(\Omega)}e^{ \alpha\tau}\right) T}.$$
	 Then, 
	 $$\|U(T)\|_{\mathcal{H}}^2+\int_{T-\tau}^{T} \|\sqrt{|a_2|}y_t(s)\|_{L^2(\Omega)}^2\, ds\leq C_T\rho^2 < \rho^2,$$
	%C_T= 2M^2(1+\tau e^{2\omega\tau})(1+\tau\|a_2\|_{L^{\infty}(\Omega)}e^{\omega\tau})
	where we have used the fact that 
	$C_T<1$.
%	We can proceed by applying a similar argument shown before on the interval
%	 $[T,2T]$, obtaining a solution on the interval $[0,2T]$. Iterating the process,
%	 we find a unique global solution to \eqref{cauchyproblem} satisfying the exponentially decaying estimate according to  \eqref{ineqexpstability}.
	 
	 We proceed by applying a similar argument as before on the interval $[T,2T]$, thereby extending the solution to $[0,2T]$. By iterating this process inductively, we construct a unique global solution to \eqref{cauchyproblem}. Moreover, at each step, the exponential decay estimate \eqref{ineqexpstability} is preserved, ensuring that the global solution satisfies this decay property.

%Toute la suite de la preuve sert uniquement à itérer afin de montrer que le résultat est vrai sur $[T,2T]$. Ainsi, on considère $t = T$ comme l’instant initial, et l'on doit alors montrer que les données initiales à $t = T$ vérifient $\leq \rho^2$.
\end{proof}
\section{The  coupled system in the presence of   indefinite damping}\label{section3}
In this section, we study the problem \eqref{indefiniteproblem}. Referring to Section \ref{section2}, one can reproduce the same procedure in the case when $\tau = 0$, and deduce that the problem \eqref{indefiniteproblem} is exponentially stable under the following condition on the coefficient $a_2$:
$$
\|a_2\|_{L^{\infty}(\Omega)} < \frac{ \alpha}{M}.
$$
The main objective of this section is to establish the exponential stability of the problem \eqref{indefiniteproblem} under a weaker condition, namely,  by restricting the assumption on the negative part of $a_2$, i.e. $a_2^{-}:=-\min \{a_2, 0\}$.

Before addressing problem \eqref{indefiniteproblem}, we first need to investigate the exponential stability of the associated coupled wave equation with a definite (i.e., nonnegative) damping:

\begin{equation}
	\label{definetedamping}
	\begin{cases}
		u_{tt}(x,t) - \Delta u(x,t) + b(x) y_{t}(x,t) + d_1(x) u_t(x,t) =0, & \text{in } \Omega \times (0, \infty), \\ 
		y_{tt}(x,t) - \Delta y(x,t) - b(x) u_{t}(x,t) + d_2(x) y_t(x,t) = 0, & \text{in } \Omega \times (0, \infty), \\ 
		u(x,t) = y(x,t) = 0, & \text{on } \partial\Omega \times (0, \infty), \\
		u(x,0) = u_{0}(x), \quad u_t(x,0) = u_1(x), & \text{in } \Omega, \\ 
		y(x,0) = y_{0}(x), \quad y_t(x,0) = y_1(x), & \text{in } \Omega, 
	\end{cases}
\end{equation}
where $d_1\in C^0(\bar{\Omega},\mathbb{R})$ and $d_2\in L^{\infty}(\Omega)$ are given functions. The coefficient $d_1$ satisfies the conditions 
$$
d_1(x) \geq 0 \quad \text{a.e in } \Omega,
$$
and
$$
d_1(x) > d_0 > 0 \quad \text{a.e in a subdomain } \omega \subset \Omega.
$$	
While the coefficient $d_2$ satisfies the condition
$$d_2(x)\geq 0\quad \text{a.e in } \Omega.$$
Next, we define the linear unbounded operator $ \mathcal{\hat{A}}: D( \mathcal{\hat{A}})\subset \mathcal{H} \rightarrow \mathcal{H}$ by
\begin{equation}\label{theoperator}
	\begin{cases}
		D(\mathcal{\hat{A}})&=\left( (H^2(\Omega) \cap H_0^1(\Omega)) \times H_0^1(\Omega) \right)^2,\\
		 \mathcal{\hat{A}} U&=(u_{2},\Delta u_1 - b y_2 - d_1 u_2, y_{2}, \Delta y_1 + b u_2-d_2y_2), \quad \forall U=(u_{1}, u_{2}, y_{1}, y_{2}) \in D( \mathcal{\hat{A}}).
	\end{cases}
\end{equation}
Using this operator, we rewrite the system \eqref{definetedamping} as an abstract Cauchy problem. Indeed, setting $U(t)=(u(t),u_t(t),y(t),y_t(t))$, one has that \eqref{definetedamping} can be formulated as 
\begin{equation}\label{cauchyproblemdefinite}
	\begin{cases}
		U^\prime(t)&=  \mathcal{\hat{A}} U(t),\quad t>0,\\
		U(0)&=U_0,
	\end{cases}
\end{equation}
where $U(0)=(u_0,u_1,y_0,y_1)$.

The existence and uniqueness result reads as follows.
\begin{theorem}\label{wellposedacta}
	For any $U_{0}=\left(u_{0}, u_{1}, y_{0}, y_{1}\right) \in \mathcal{H}$, there exists a unique solution $U=(u,u_t,y,y_t)\in C^{0}([0,\infty); \mathcal{H})$ of system \eqref{cauchyproblemdefinite}. Moreover, if $U_{0}=\left(u_{0}, u_{1}, y_{0}, y_{1}\right) \in D( \mathcal{\hat{A}})$, then
	$$
	U \in C^{0}([0,\infty); D( \mathcal{\hat{A}})) \cap C^{1}([0,\infty); \mathcal{H}).
	$$
\end{theorem}
\begin{proof}
	We will prove that the operator $ \mathcal{\hat{A}}$ defined in \eqref{theoperator} generates a contraction semi-group on the Hilbert
	space $\mathcal{H}$. To this end, we show that the unbounded operator $ \mathcal{\hat{A}}$ is maximal dissipative on $\mathcal{H}$.
	According to \cite[Proposition 2.2.6]{Cazenave1998},
	it is sufficient to prove that $ \mathcal{\hat{A}}$ is dissipative and that $I-  \mathcal{\hat{A}}$ is surjective, where
	$$
	I= \operatorname{diag}(Id,Id,Id,Id).
	$$
	
	\underline{$ \mathcal{\hat{A}}$ is dissipative.} Let $U=(u_1,u_2,y_1,y_2)\in D( \mathcal{\hat{A}})$. Then, by integrating by parts and  due to the fact that $d_1(x)\geq 0$ and $d_2(x)\geq 0$, we obtain
	\begin{align*}
		\langle  \mathcal{\hat{A}}U,U\rangle_{\mathcal{H}} =& \langle \big(u_{2},\Delta u_{1}-b y_{2}-d_1u_2, y_{2},  \Delta y_1 + b u_2-d_2y_2\big),\big(u_1,u_2,y_1,y_2\big) \rangle_{\mathcal{H}}\\
		=&\int_{\Omega}\Big(\nabla u_1\nabla u_2+(\Delta u_{1}-b y_{2}-d_1u_2)u_2+\nabla y_1\nabla y_2+( \Delta y_1 + b u_2-d_2y_2)y_2\Big) dx\\
		=&-\int_{\Omega}d_1(x) u_2^2(x)\,dx-\int_{\Omega}d_2(x)y_2^2(x)\,dx\\
		\leq &0.	
	\end{align*}
	Therefore the operator $ \mathcal{\hat{A}}$ is dissipative.
	
	\underline{$I- \mathcal{\hat{A}}$ is surjective.} Given $F=\left(f_{1}, f_{2}, f_{3}, f_{4}\right) \in \mathcal{H}$,
	we seek  $U=\left(u_{1}, u_{2}, y_{1}, y_{2}\right) \in D( \mathcal{\hat{A}})$ such that
	\begin{equation}\label{probdisspdege}
		U- \mathcal{\hat{A}} U=F \Leftrightarrow
		\begin{cases}
			u_{1}-u_{2} &=f_{1}, \\
			u_{2}-\Delta u_1+b y_{2}+d_1u_2&=f_{2}, \\
			y_{1}-y_{2} &=f_{3}, \\
			y_{2}-\Delta y_1-b u_{2}+d_2y_2 &=f_{4}.
		\end{cases}
	\end{equation}
	We suppose that we have found $u_1$ and $y_1$ with the right regularity. Then, we set
	\begin{equation}
		\label{u2byu1}
		u_2= u_1-f_1 \quad \text{and }\quad y_2= y_1-f_3.
	\end{equation}
	Eliminating $u_{2}$ and $y_{2}$ in \eqref{probdisspdege}, we get the following system
	\begin{equation}
		\label{prbfidege}
		\begin{cases}
			u_{1}-\Delta u_1+ b y_{1}+ d_1 u_1&= (1+d_1)f_{1}+f_{2}+b f_{3}, \\
			y_{1}-\Delta y_1- b u_{1}+d_2y_1&=(1+d_2)f_{3}-b f_{1}+f_{4}.
		\end{cases}
	\end{equation}
	Now, we consider the bilinear form $\Gamma:\left(H_{0}^{1}(\Omega)\times H_{0}^{1}(\Omega)\right)^{2} \rightarrow \mathbb{R}$ given by
	$$
	\begin{aligned}
		\Gamma\left(\left(u_{1}, y_{1}\right),\left(\phi_{1}, \phi_{2}\right)\right)=&\int_{\Omega}(u_{1} \phi_{1}+\nabla u_{1} \nabla\phi_{1} +y_{1} \phi_{2}+\nabla y_{1} \nabla\phi_{2})d x+\int_{\Omega}d_1u_1\phi_1\,dx\\
		&+\int_{\Omega}d_2y_1\phi_2\,dx+ \int_{\Omega}b \left(y_{1} \phi_{1}-u_{1} \phi_{2}\right)\,dx
	\end{aligned}
	$$
	and the linear form $L:H_{0}^{1}(\Omega)\times H_{0}^{1}(\Omega)\rightarrow \mathbb{R}$ given by
	$$
	L\left(\phi_{1}, \phi_{2}\right)=\int_{\Omega}\left((1+d_1)f_{1}+f_{2}+b f_{3}\right) \phi_{1}\,dx + \int_{\Omega}\left((1+d_2)f_{3}-b f_{1}+f_{4}\right) \phi_{2}\,dx.
	$$
	In view of Poincaré inequality, $\Gamma$ is a continuous bilinear form on $H_{0}^{1}(\Omega)\times H_{0}^{1}(\Omega)$ and $L$ is a continuous linear functional on $H_{0}^{1}(\Omega)\times H_{0}^{1}(\Omega)$. Moreover, it is easy to check that $\Gamma$ is also coercive on $H_{0}^{1}(\Omega)\times H_{0}^{1}(\Omega)$. Indeed, let $(u_1,y_1)\in H_{0}^{1}(\Omega)\times H_{0}^{1}(\Omega)$  
	$$\begin{aligned}
		\Gamma((u_1,y_1),(u_1,y_1))&=\int_{\Omega}\left(u_1^2+|\nabla u_1|^2+y_1^2+|\nabla y_1|^2\right)\,dx+ \int_{\Omega}d_1u_1^2\,dx+ \int_{\Omega}d_2y_1^2\,dx\\
		&\geq \|(u_1,y_1)\|_{H^{1}_{0}(\Omega)\times H^{1}_{0}(\Omega)}^2.
	\end{aligned}$$
	By applying the Lax–Milgram theorem, we deduce that there exists a unique solution $\left(u_{1}, y_{1}\right) \in H_{0}^{1}(\Omega)\times H_{0}^{1}(\Omega)$ of
	\begin{equation}\label{pv}
		\Gamma\left(\left(u_{1}, y_{1}\right),\left(\phi_{1}, \phi_{2}\right)\right)= L\left(\phi_{1}, \phi_{2}\right)\, \text{for all } (\phi_1,\phi_2)\in H_{0}^{1}(\Omega)\times H_{0}^{1}(\Omega).
	\end{equation}
	If $(\phi_1,\phi_{2})\in \mathcal{D}(\Omega)\times \mathcal{D}(\Omega)$, then $(u_1,y_1)$ solves \eqref{prbfidege} in $\mathcal{D}^{\prime}(\Omega)\times \mathcal{D}^{\prime}(\Omega)$ and so $(u_1,y_1)\in \left(H^{2}(\Omega)\right)^{2}$ (See \cite[Chapter 5]{Allaire}).
	Now, take $(u_2,y_2):=(u_1-f_1, y_1-f_3)$; then $(u_2,y_2)\in \left(H_{0}^{1}(\Omega)\right)^{2}$.
	Thus, we have found $U=\left(u_{1}, u_{2}, y_{1}, y_{2}\right) \in D( \mathcal{\hat{A}})$, which satisfies \eqref{probdisspdege}. Consequently $ I- \mathcal{\hat{A}}$ is surjective. Conclusively, the
	Lumer–Phillips Theorem implies the operator
	$ \mathcal{\hat{A}}$ generates a strongly continuous semigroup of contraction
	in $\mathcal{H}$. Consequently, the well-posedness result follows from the Hille-Yosida theorem (see \cite[Theorem 4.5.1]{barbu} or  \cite[Theorem A.11]{Coron2007}).% (see \cite[Theorem 4.5.1]{barbu} or  \cite[Theorem A.11]{Coron2007}).
\end{proof}
%%%%%%%%%%%%%%%%%%%%%%%%%%%%%%%%%%%%%%%%%%%%%%%%%%%%%%%%%%%%%%%%%%%%%%%%%%%%%%%%%%%
	Now, we study the exponential stability of system \eqref{definetedamping}.
To this aim we first define its corresponding energy by
\begin{equation}\label{energysystem}
	\begin{aligned}
		\mathcal{E}(t)&:=\frac{1}{2} \int_{\Omega}\left\{u_{t}^{2}(x,t)+|\nabla u(x,t)|^2+y_{t}^{2}(x,t)+|\nabla y(x,t)|^2\right\}\,dx,\quad \forall t \geq 0.
	\end{aligned}
\end{equation}

Direct computations  allow to show that  the energy   exhibits a decay with respect to time. More precisely, we have the following result.
The energy $\mathcal{E}$ associated to the solution of system \eqref{definetedamping}  satisfies
	\begin{equation}\label{decergi}
		\begin{aligned}
			\mathcal{E}^{\prime}(t)&= -\int_{\Omega}d_1(x)u_t^2(x,t)\,dx-\int_{\Omega}d_2(x)y_t^2(x,t)\,dx,\quad \forall t \geq 0.
		\end{aligned}
	\end{equation}	
	Consequently,  the energy $\mathcal{E}(.)$ is nonincreasing.

%%%%%%%%%%%%%%%%%%%%%%%%%%%%%%%%%%%%%%%%%%%%%%%%%%%%%%%%%
%%%%%%%%%%%%%%%%%%%%%%%%%% Determining the rate of decay
%%%%%%%%%%%%%%%%%%%%%%%%%%%%%%%%%%%%%%%%%%%%%%%%%%%%%%%%%
%%%%%%%%%%%%%%%%%%%%%%%%%%%%%%%%%%%%%%%%%%%%%%%%%%%%%%%%%%

 We can prove the following exponential stability result for the  problem \eqref{definetedamping}.
\begin{theorem}
	\label{Theoremexpstabilityd}
Assume that either Assumption \ref{assumptionsgerbi} or Assumption \ref{assumptionsalabau} holds. There exist two positive constants $\hat{M}, \hat{\alpha}$ such that
	\begin{equation}
		\label{expstability}
		\mathcal{E}(t) \leq \hat{M} e^{-\hat{\alpha} t} \mathcal{E}(0), \quad t>0,
	\end{equation}
	for any solution to the problem \eqref{definetedamping}.
\end{theorem}
\begin{proof}
	As in \cite{nicaise2006stability}, by following \cite{enrike1990exponential}, we can decompose the solution $(u,y)$ of \eqref{definetedamping}  as $u=w+\tilde{w}$ and $y=\theta+\tilde{\theta},$ 
	where $(w,\theta)$ is the solution to the problem 
	\begin{equation}\label{problemobs}
		\begin{cases}
			w_{tt}(x,t)-\Delta w(x,t)+b \theta_t(x,t)+d_1w_t(x,t)=0, & \text{in } \Omega \times (0, \infty), \\ 
			\theta_{tt}(x,t)-\Delta \theta (x,t)-b w_{t}(x,t)=0, & \text{in }\Omega \times (0, \infty), \\ 
			w(x,t)=\theta(x,t)=0, & \text {on }\partial\Omega \times (0, \infty),\\
			w(x,0)=u_{0}(x), \quad w_t(x,0)=u_1(x), & \text{in } \Omega, \\ 
			\theta(x,0)=y_0(x), \quad \theta_t(x,0)=y_1(x),  & \text{in } \Omega,
		\end{cases}
	\end{equation}
	and $(\tilde{w},\tilde{\theta})$ solution to the problem
	\begin{equation}\label{problemdecomp}
		\begin{cases}
			\tilde{w}_{tt}(x,t)-\Delta \tilde{w}(x,t)+b \tilde{\theta}_t(x,t)+d_1\tilde{w}_t(x,t)=0, & \text{in } \Omega \times (0, \infty), \\ 
			\tilde{\theta}_{tt}(x,t)-\Delta \tilde{\theta}(x,t)-b \tilde{w}_{t}(x,t)=-d_2(x)y_t(x,t), & \text{in }\Omega \times (0, \infty), \\ 
			\tilde{w}(x,t)=	\tilde{\theta}(x,t)=0, & \text {on }\partial\Omega \times (0, \infty),\\
			\tilde{w}(x,0)=0, \quad \tilde{w}_t(x,0)=0, & \text{in } \Omega, \\ 
			\tilde{\theta}(x,0)=0, \quad \tilde{\theta}_t(x,0)=0,  & \text{in } \Omega.
		\end{cases}
	\end{equation}	
	Thus, as a direct consequence of \cite[Proposition 2]{Haraux}, it follows that under either Assumption \ref{assumptionsgerbi} or Assumption \ref{assumptionsalabau}, the exponential stability property given by \eqref{exp_decay} implies the existence of a constant time $T_0 > 0$ such that, for every $T > T_0$, the following holds: 
	\begin{equation}
		\label{majE0}
\begin{aligned}
	\mathcal{E}(0)\leq  C \int_{0}^{T}\int_{\Omega}d_1(x)w_t^2(x,t)\,dx\,dt\leq 2C \int_\Omega d_1(x) \int_0^T  \left(u_{t}^2(x,t)+\tilde{w}_t^2(x, t)\right) d t d x.
	\end{aligned}
		\end{equation}
	Now, observe that from \eqref{problemdecomp}, it follows that
	$$\begin{aligned}
		\frac{1}{2} \frac{d}{d t} &  \int_{\Omega}\left\{\tilde{w}_t^2(x,t)+|\nabla \tilde{w}(x,t)|^2+\tilde{\theta}_t^2(x,t)+|\nabla \tilde{\theta}(x,t)|^2\right\} d x\\
		&=\int_{\Omega}\left\{\tilde{w}_t(x,t)\tilde{w}_{tt}(x,t)+\nabla \tilde{w}(x,t)\nabla \tilde{w}_t(x,t)+\tilde{\theta}_t(x,t)\tilde{\theta}_{tt}(x,t)+\nabla \tilde{\theta}(x,t)\nabla \tilde{\theta}_t(x,t)\right\}dx \\
		& =\int_{\Omega}\left(-d_1(x)\tilde{w}_t^2(x,t)- d_2(x)\tilde{\theta}_t(x,t)y_t(x,t)\right)  d x.
	\end{aligned}$$
	Integrating with respect to time on $[0, t]$, for $t \in(0, T]$, and using the fact that the initial data in \eqref{problemdecomp} is set to zero, we have
	$$
	\begin{array}{l}
\displaystyle{
		\frac{1}{2} \int_{\Omega}\left\{\tilde{w}_t^2(x,t)+|\nabla \tilde{w}(x,t)|^2+\tilde{\theta}_t^2(x,t)+|\nabla \tilde{\theta}(x,t)|^2\right\} d x }\\
\hspace{2 cm} \displaystyle{=  - \int_{\Omega}d_1(x)\int_0^t  \tilde{w}_t^2(x,t) d t d x -\int_{\Omega}d_2(x)\int_0^t \tilde{\theta}_t(x,t) y_t(x,t) d t d x}\\
\hspace{2 cm} \displaystyle{ \le \int_{\Omega}d_2(x)\int_0^T  \vert \tilde{\theta}_t(x,t) y_t(x,t)\vert d t d x.}
	\end{array}
	$$	
	By integrating with respect to time, we deduce that
	$$
	\begin{array}{l}
\displaystyle{\int_0^T \int_{\Omega} \left\{\tilde{w}_t^2(x,t)+\tilde{\theta}_t^2(x,t)\right\} d x d t} \\
\displaystyle{\hspace{2cm}
		\leq  \frac{1}{2}\int_{\Omega}\int_{0}^{T} \tilde{\theta}_t^2(x,t)  dt dx+2T^2\|d_2\|_{L^{\infty}(\Omega)}\int_{\Omega}d_2(x)\int_{0}^{T} y_t^2(x,t)  d t d x,}
		\end{array}
	$$
	which implies that
	$$
	\begin{aligned}
		\int_0^T \int_{\Omega} \tilde{w}_t^2(x,t) d x d t \leq& 2T^2\|d_2\|_{L^{\infty}(\Omega)} \int_{\Omega}d_2(x)\int_{0}^{T} y_t^2(x,t)  d t d x.
	\end{aligned}
	$$ 
	This combined with \eqref{majE0} yield that
	$$\mathcal{E}(0)\leq C_0 \left(\int_{\Omega}d_1(x)\int_{0}^{T} u_t^2(x,t) d x d t+\int_{\Omega}d_2(x)\int_{0}^{T} y_t^2(x,t) d x d t\right),$$
	where $C_0$ is a positive constant.

	From \eqref{decergi}, we have
	$$
	\mathcal{E}(T)-\mathcal{E}(0) =- \left(\int_{\Omega}d_1(x) \int_0^Tu_t^2(x, t) d t dx+\int_{\Omega}d_2(x) \int_0^Ty_t^2(x, t) d t dx\right).
	$$
	
	Hence, we obtain
	$$
	\mathcal{E}(T) \leq \mathcal{E}(0) \leq C_0 \left(\int_{\Omega}d_1(x) \int_0^Tu_t^2(x, t) d t dx+\int_{\Omega}d_2(x) \int_0^Ty_t^2(x, t) d t dx\right) = C_0 (\mathcal{E}(0)-\mathcal{E}(T)),
	$$
	then
	$$
	\mathcal{E}(T) \leq \hat{C} \mathcal{E}(0)
	$$
	with $\hat{C}<1$. 
	
	This easily implies the stability estimate \eqref{expstability}, since our system \eqref{definetedamping} is invariant by translation and the energy $\mathcal{E}$ is decreasing.
\end{proof}
Returning to the problem \eqref{indefiniteproblem}, and using the same notations as those introduced in Section \ref{section2}, we define the linear operator $\mathcal{\tilde{A}}: D(\mathcal{\tilde{A}}) \subset \mathcal{H} \to \mathcal{H}$ by
$$
D(\mathcal{\tilde{A}}) = \left( (H^2(\Omega) \cap H_0^1(\Omega)) \times H_0^1(\Omega) \right)^2, \quad 
\mathcal{\tilde{A}} U = \begin{bmatrix} 
	u_2 \\ 
	\Delta u_1 - b y_2 - a_1 u_2 \\ 
	y_2 \\  
	\Delta y_1 + b u_2 - a_2^+ y_2 
\end{bmatrix},
$$
where $a_2^+$ denotes the positive part of the function $a_2$, defined by
$$
a_2^+(x) := \max\{a_2(x), 0\}, \quad \forall x \in \Omega.
$$

By applying Theorem \ref{Theoremexpstabilityd} with $d_1 = a_1$ and $d_2 = a_2^+$, we deduce that the operator $\mathcal{\tilde{A}}$ is non-negative and generates a contraction semigroup $(\tilde{S}(t))_{t \geq 0}$ on $\mathcal{H}$, which is exponentially stable under either Assumption \ref{assumptionsgerbi} or Assumption \ref{assumptionsalabau}. More precisely, there exist constants $\tilde{M} > 0$ and $\tilde{\alpha} > 0$ such that
$$
\| \tilde{S}(t) \|_{\mathcal{L}(\mathcal{H})} \leq \tilde{M} e^{-\tilde{\alpha} t}, \quad \forall t \geq 0.
$$

Consequently, the  coupled system \eqref{indefiniteproblem} can be reformulated as the following abstract Cauchy problem:
\begin{equation}
	\label{cauchyproblemindefinite}
	\begin{cases}
		\dot{U}(t) = \mathcal{\tilde{A}} U(t) + a_2^{-} P U(t) + \mathcal{F}(U(t)), & t > 0, \\
		U(0) = U_0, &
	\end{cases}
\end{equation}
where $P \in \mathbb{R}^{4 \times 4}$ is the linear projection operator defined by
$$
P := \begin{pmatrix}
	0 & 0 & 0 & 0 \\
	0 & 0 & 0 & 0 \\
	0 & 0 & 0 & 0 \\
	0 & 0 & 0 & 1
\end{pmatrix}.
$$

 Observe that, if  $U_0 \in \mathcal{H}$, then the following Duhamel formula holds: 
 \begin{equation}
 	\label{Duhamelindefinite}
 	U(t)=\tilde{S}(t) U_0+\int_0^t \tilde{S}(t-s)  a_2^{-} PU(s) d s+\int_0^t \tilde{S}(t-s) \mathcal{F}(U(s)) d s.
 \end{equation}
 Using  \cite[Proposition 4.3.3]{Cazenave1998}, we can establishing the local existence and uniqueness of solutions, as stated in the following lemma.
 \begin{lemma}
 	\label{lemmalocalexistenceindefinite}
 	Let $U_0\in\mathcal{H}$. Then, the system \eqref{cauchyproblemindefinite} has  a unique continuous local solution $U$ defined on
 	a time interval $[0,\delta)$.
 \end{lemma}
 
Let us introduce the following energy functional associated to system \eqref{cauchyproblemindefinite}:
 \begin{definition}
 	Let $(u,y)$  be a mild solution of \eqref{indefiniteproblem}. Let us define its energy by
 	\begin{equation}
 		\begin{aligned}
 			\tilde{E}(t)&=\frac{1}{2}\int_{\Omega}\left\{u_t^2(x,t)+|\nabla u(x,t)|^2+y_t^2(x,t)+|\nabla y(x,t)|^2\right\}\,dx\\
 			&-\int_{\Omega}F_1(u(x,t))\,dx-\int_{\Omega}F_2(y(x,t))\,dx.
 		\end{aligned}
 	\end{equation}
 \end{definition}
 By arguing as in the proof of Lemma \ref{lemmaminenergy1}, we can establish the following result concerning the indefinite damping case.
 \begin{lemma}
 	\label{lemmaminenergy1indefinite}
 	Assume that Hypothesis \ref{nonlinearassumption} holds.  Let $U$ be a non-trivial solution of the problem \eqref{cauchyproblemindefinite}, defined on the interval $[0, T).$ If the initial data satisfy
 	$$
 	h_1(\|\nabla u_0\|_{L^2(\Omega)}) < \frac{1}{2} \quad \text{and} \quad h_2(\|\nabla y_0\|_{L^2(\Omega)}) < \frac{1}{2},
 	$$
 	then the initial energy is strictly positive, i.e., $\tilde{E}(0) > 0$.
 \end{lemma}
 We also have the following result.
 \begin{proposition}
 	\label{Propmajenergyindefinite}
 	Let $(u,y)$ be a solution to \eqref{indefiniteproblem}. If $\tilde{E}(t)\geq \frac{1}{4}\|y_t(t)\|^2_{L^2(\Omega)}$, for any $t\geq 0$, then 
 	\begin{equation}
 		\label{estimateenergyindefinite}
 		\tilde{E}(t)\leq \tilde{C}(t)\tilde{E}(0),
 	\end{equation}
 	for any $t\geq 0$, where 
 	\begin{equation}
 		\label{functionCindefinite}
 		\tilde{C}(t)=e^{4\|a_2^{-}\|_{L^{\infty}(\Omega)}t}.
 	\end{equation}
 \end{proposition}
 \begin{proof}
 	Differentiating formally $\tilde{E}$ with respect to $t$, we
 	obtain
 	$$\begin{aligned}
 		\tilde{E}^{\prime}(t)&=\int_{\Omega}\left\{u_t(x,t) u_{tt}(x,t)+\nabla u(x,t)\nabla u_t(x,t) +y_t(x,t)y_{tt}(x,t)+\nabla y(x,t)\nabla y_t(x,t)\right\}\,dx\\
 		&-\int_{\Omega}f_1(u(x,t))u_t(x,t)\,dx-\int_{\Omega}f_2(y(x,t))y_t(x,t)\,dx.
 	\end{aligned}$$
 	By using the equations satisfied by $u$ and $y$ in  \eqref{coupledstabilityproblem} and Green formula, we get
 	$$\begin{aligned}
 		\tilde{E}^{\prime}(t)&=\int_{\Omega}\left\{u_t(x,t) \left(\Delta u(x,t)-by_t(x,t)-a_1(x)u_t(x,t)+f_1(u(x,t))\right)+\nabla u(x,t)\nabla u_t(x,t)\right.\\
 		&\left.+y_t(x,t)\left(\Delta y(x,t)+b u_{t}(x,t)-a_2(x)y_t(x,t)+f_2(y(x,t))\right)+\nabla y(x,t)\nabla y_t(x,t)\right\}\,dx\\
 		&-\int_{\Omega}u_t(x,t)f_1(u(x,t))\,dx-\int_{\Omega}y_t(x,t)f_2(y(x,t))\,dx\\
 		=&-\int_{\Omega}a_1(x)u_t^2(x,t)\,dx-
 		\int_{\Omega}a_2(x)y_t^2(x,t)\,dx\\
 		=&-\int_{\Omega}a_1(x)u_t^2(x,t)\,dx-
 		\int_{\Omega}a_2^{+}(x)y_t^2(x,t)\,dx+\int_{\Omega}a_2^{-}(x)y_t^2(x,t)\,dx\\
 		&\leq\int_{\Omega}a_2^{-}(x)y_t^2(x,t)\,dx\leq \Vert a_2^{-}\Vert_{L^{\infty}(\Omega)}\|y_t(t)\|_{L^2(\Omega)}^2.
 	\end{aligned}$$
 	Since we have assumed $\tilde{E}(t)\geq \frac{1}{4}\|y_t(t)\|_{L^2(\Omega)}^2$
 	for any $t\geq 0$, we obtain
 	$$\tilde{E}^{\prime}(t)\leq 4\|a_2^{-}\|_{L^{\infty}(\Omega)}\tilde{E}(t).$$
 	By Gronwall inequality we then obtain \eqref{estimateenergyindefinite}.
 \end{proof}
Using Proposition \ref{Propmajenergyindefinite} and adapting the arguments used in the proof of Lemma \ref{lemmaminenergy}, we obtain the following result concerning the lower bound of the energy in the presence of indefinite damping.
\begin{lemma}
	\label{lemmaminenergyindefinite}
	Assume that Hypothesis \ref{nonlinearassumption} holds. Let $U$ be a non-trivial solution of  problem \eqref{cauchyproblemindefinite}, defined on the interval $[0, \delta)$ and fix a time $T\ge \delta$. 
	
	Suppose the initial data $(u_0, u_1, y_0, y_1)$ satisfy the following conditions:
	$$
	h_1\left(\|\nabla u_0\|_{L^2(\Omega)}\right) < \frac{1}{2}, \quad 
	h_2\left(\|\nabla y_0\|_{L^2(\Omega)}\right) < \frac{1}{2}, \quad 
	\text{and} \quad 
	h_i\left(2\sqrt{\tilde{C}(T)\tilde{E}(0)}\right) < \frac{1}{2} \quad \text{for } i = 1, 2,
	$$
	where $\tilde{C}(\cdot )$ is the function defined in \eqref{functionCindefinite}.
	
	Then, the energy $\tilde{E}(t)$ satisfies the following strict lower bound for all $t \in [0, \delta)$:
	\begin{equation}
		\label{minorationenergyindefinite}
		\tilde{E}(t) > \frac{1}{4}\|u_t(t)\|^2_{L^2(\Omega)} + \frac{1}{4}\|\nabla u(t)\|^2_{L^2(\Omega)} + \frac{1}{4}\|y_t(t)\|^2_{L^2(\Omega)} + \frac{1}{4}\|\nabla y(t)\|^2_{L^2(\Omega)}.
	\end{equation}
	
	In particular, one has:
	$$
	\tilde{E}(t) > \frac{1}{4} \|U(t)\|^2_{\mathcal{H}}, \quad \forall t \in [0, \delta).
	$$
\end{lemma}

 Next, we establish an exponential decay result for the abstract system \eqref{cauchyproblemindefinite} under a weaker condition on the damping coefficient $a_2$. More precisely, we impose a smallness assumption only on the negative part of $a_2$. To this end, we introduce the following hypothesis:
 
 \begin{assumptions}
 	\label{assumptionona2indefinite}
 	The function $a_2 \in L^{\infty}(\Omega)$ satisfies
 	$$
 	\|a_2^{-}\|_{L^{\infty}(\Omega)} < \frac{\tilde{\alpha}}{\tilde{M}}.
 	$$
 \end{assumptions}
 
 Under this assumption, we prove that the solution to \eqref{cauchyproblemindefinite} decays exponentially in time. This is stated precisely in the following theorem.
 \begin{theorem}
 	\label{Theoremexpstabilityindefinite}
 	Assume that either Assumption \ref{assumptionsgerbi} or Assumption \ref{assumptionsalabau} holds, together with Hypotheses \ref{nonlinearassumption} and \ref{assumptionona2indefinite}. Let $U \in C([0, T); \mathcal{H})$ be a solution of the problem \eqref{cauchyproblemindefinite} satisfying the uniform bound
 	$$
 	\|U(t)\|_{\mathcal{H}} \leq C_\rho, \quad \forall t \in [0, T),
 	$$
 	for some constant $C_\rho > 0$ such that
 	$$
 	L(C_\rho) < \frac{\tilde{\alpha} - \tilde{M} \|a_2^{-}\|_{L^{\infty}(\Omega)}}{\tilde{M}},
 	$$
 	where $L(\cdot)$ is the Lipschitz constant of the nonlinearity as defined in Hypothesis \ref{nonlinearassumption}.
 	
 	Then, the following exponential decay estimate holds:
 	\begin{equation}
 		\label{ineqexpstabilityindefinite}
 	\|U(t)\|_{\mathcal{H}} \leq \tilde{M} \|U_0\|_{\mathcal{H}} \, e^{-\left(\tilde{\alpha} - \tilde{M} \|a_2^{-}\|_{L^{\infty}(\Omega)} - \tilde{M} L(C_\rho)\right)t}, \quad \forall t \in [0, T).
 		\end{equation}
 \end{theorem}
 \begin{proof}
 	Applying Duhamel's formula \eqref{Duhamelindefinite}, we obtain
 	\begin{equation}
 		\label{ineqDuhamelindefinite}
 		\begin{aligned}
 			\|U(t)\|_{\mathcal{H}} \leq & \tilde{M} e^{-\tilde{\alpha} t}\left\|U_0\right\|_{\mathcal{H}}+\tilde{M} e^{-\tilde{\alpha} t} \int_0^t e^{ \tilde{\alpha} s} \|a_2^{-}\|_{L^{\infty}(\Omega)}  \|P U(s)\|_{\mathcal{H}} d s \\
 			& +\tilde{M} e^{- \tilde{\alpha} t}L\left(C_\rho\right)  \int_0^t e^{ \tilde{\alpha} s}\|U(s)\|_{\mathcal{H}} d s,
 		\end{aligned}
 	\end{equation}
 	where we have used the assumption $\|F(U(t))\|_{\mathcal{H}} \leq L(C_\rho)\|U(t)\|_{\mathcal{H}}$ for all $t \geq 0$.
 		
 	 It follows that
 		\begin{equation*}
 			\begin{aligned}
 				 e^{ \tilde{\alpha} t}\|U(t)\|_{\mathcal{H}} \leq\tilde{M} \left\|U_0\right\|_{\mathcal{H}}+\left(\tilde{M}\|a_2^{-}\|_{L^{\infty}(\Omega)}+\tilde{M}L\left(C_\rho\right) \right) \int_0^t e^{ \tilde{\alpha} s}\|U(s)\|_{\mathcal{H}} d s.
 			\end{aligned}
 		\end{equation*}
 		By Gronwall's inequality, we obtain
 		\begin{equation*}
 		 e^{ \tilde{\alpha} t}\|U(t)\|_{\mathcal{H}}\leq\tilde{M} \left\|U_0\right\|_{\mathcal{H}}e^{\left(\tilde{M}\|a_2^{-}\|_{L^{\infty}(\Omega)}+\tilde{M}L\left(C_\rho\right) \right) t}.
 		\end{equation*}
 		As before, using assumption \ref{assumptionona2indefinite}, we obtain \eqref{ineqexpstabilityindefinite}.
 \end{proof}
We are now in a position to establish the global well-posedness and the exponential decay estimates for solutions to \eqref{indefiniteproblem} corresponding to small initial data.
 
 \begin{theorem}
 	Assume that either Assumption \ref{assumptionsgerbi} or Assumption \ref{assumptionsalabau} holds, together with Hypotheses \ref{nonlinearassumption} and \ref{assumptionona2indefinite}. Let the initial data $U_0 \in \mathcal{H}$ satisfy
 	\begin{equation}
 		\label{smallnessinitialdataindefinite}
 		\|U_0\|_{\mathcal{H}} < \rho,
 	\end{equation}
 	for a sufficiently small constant $\rho > 0$. Then, the Cauchy problem \eqref{cauchyproblemindefinite} admits a unique global solution which decays exponentially in time according to estimate \eqref{ineqexpstabilityindefinite}.
 \end{theorem}
 
 \begin{proof}
 	 Let us fix $T > 0$, such that 
 	\begin{equation}
 		\label{conditiononT}
 	\tilde{M}^2 \rho^2 e^{-(\tilde{\alpha} - \tilde{M} \|a_2^{-}\|_{L^\infty(\Omega)}) T} < 1.
 		\end{equation}
 	Moreover let   $\rho > 0$ be such that
 	$$
 	\rho \leq \frac{1}{2\sqrt{\tilde{C}(T)}} \min_{i=1,2} \left\{ h_i^{-1}\left( \frac{1}{2} \right) \right\}.
 	$$ 
 	We consider initial data satisfying
 	\begin{equation}
 		\label{assymptiononinitialdataindefinite}
 		\|U_0\|_{\mathcal{H}} < \rho.
 	\end{equation}
 	By Lemma \ref{lemmalocalexistenceindefinite}, there exists a local solution $(u,y)$ to \eqref{indefiniteproblem} on some interval $[0,\delta)$, with $\delta > 0$. Without loss of generality, we may assume $\delta < T$.
 	
 	Using the monotonicity of $h_i$ and the fact that $\tilde{C}(T) > 1$, we obtain
 	$$
 	h_1(\|\nabla u_0\|_{L^2(\Omega)}) \leq h_1(\rho) < \frac{1}{2}, \quad h_2(\|\nabla y_0\|_{L^2(\Omega)}) \leq h_2(\rho) < \frac{1}{2}.
 	$$
 	Therefore, by Lemma \ref{lemmaminenergy1indefinite}, we have $\tilde{E}(0) > 0$. Moreover, from estimate \eqref{estimatenonlinearterm}, it follows that
 	\begin{equation}
 		\label{ineqhelp0indefinite}
 		\tilde{E}(0) \leq \frac{1}{2}\|u_1\|_{L^2(\Omega)}^2 + \frac{3}{4}\|\nabla u_0\|_{L^2(\Omega)}^2 + \frac{1}{2}\|y_1\|_{L^2(\Omega)}^2 + \frac{3}{4}\|\nabla y_0\|_{L^2(\Omega)}^2 \leq \rho^2.
 	\end{equation}
 	This implies, for $i=1,2$,
 	$$
 	h_i\big(2\sqrt{\tilde{C}(T)\tilde{E}(0)}\big) < h_i\left(h_i^{-1}\left(\frac{1}{2}\right)\right) = \frac{1}{2}.
 	$$
 	Hence, applying Lemma \ref{lemmaminenergyindefinite}, we obtain the energy lower bound
 	\begin{equation}
 		\label{ineqhelp1indefinite}
 		\tilde{E}(t) > \frac{1}{4} \|u_t(t)\|^2_{L^2(\Omega)} + \frac{1}{4} \|\nabla u(t)\|^2_{L^2(\Omega)} + \frac{1}{4} \|y_t(t)\|^2_{L^2(\Omega)} + \frac{1}{4} \|\nabla y(t)\|^2_{L^2(\Omega)} > 0, \quad \forall t \in [0,\delta).
 	\end{equation}
 	
 	From Proposition \ref{Propmajenergyindefinite} and the fact that $\delta < T$, we deduce
 	\begin{equation}
 		\label{ineqhelp2indefinite}
 		\tilde{E}(t) \leq \tilde{C}(T) \tilde{E}(0), \quad \forall t \in [0,\delta).
 	\end{equation}
 	Combining \eqref{ineqhelp1indefinite} and \eqref{ineqhelp2indefinite}, we obtain the bounds
 	$$
 \frac{1}{4}	\|\nabla u(t)\|_{L^2(\Omega)}^2 \leq  \tilde{E}(t) \leq  \tilde{C}(t) \tilde{E}(0), \quad 
 	\frac{1}{4}\|\nabla y(t)\|_{L^2(\Omega)}^2 \leq  \tilde{E}(t) \leq  \tilde{C}(t) \tilde{E}(0).
 	$$
 	Hence, the solution can be extended beyond $\delta$ up to $t = T$. In particular, for $t = T$, we have
 $$h_1(\|\nabla u(T)\|_{L^2(\Omega)})< h_1(2\sqrt{\tilde{C}(T)\tilde{E}(0)})< \frac{1}{2}$$
 and 
 $$h_2(\|\nabla y(T)\|_{L^2(\Omega)})< h_2(2\sqrt{\tilde{C}(T)\tilde{E}(0)})< \frac{1}{2}.$$
 	From \eqref{ineqhelp1indefinite} and \eqref{ineqhelp0indefinite}, we also deduce
 	$$
 	\frac{1}{4} \|U(t)\|_{\mathcal{H}}^2 \leq \tilde{E}(t) \leq \tilde{C}(T) \tilde{E}(0) \leq \tilde{C}(T)\rho^2, \quad \forall t \in [0,T],
 	$$
 	which yields
 	$$
 	\|U(t)\|_{\mathcal{H}} \leq C_\rho := 2\sqrt{\tilde{C}(T)}\rho, \quad \forall t \in [0,T].
 	$$
 	
 	Now, by possibly taking $\rho$ smaller, we can ensure that
 	$$
 	L(C_\rho) < \frac{\tilde{\alpha} - \tilde{M} \|a_2^{-}\|_{L^\infty(\Omega)}}{2\tilde{M}}.
 	$$
 	Hence, Theorem \ref{Theoremexpstabilityindefinite} implies that the solution satisfies the exponential decay estimate
 	\begin{equation}
 		\label{majuintindefinite}
 		\|U(t)\|_{\mathcal{H}} \leq \tilde{M} \|U_0\|_{\mathcal{H}} e^{-\frac{\tilde{\alpha} - \tilde{M} \|a_2^{-}\|_{L^\infty(\Omega)}}{2} t}, \quad \forall t \in [0,T].
 	\end{equation}
 	In particular, we get
 	$$
 	\|U(t)\|_{\mathcal{H}}^2 \leq \tilde{M}^2 \rho^2 e^{-(\tilde{\alpha} - \tilde{M} \|a_2^{-}\|_{L^\infty(\Omega)}) t}, \quad \forall t \in [0,T].
 	$$
 	Then,
 	$$
 	\|U(T)\|_{\mathcal{H}} \leq \rho,
 	$$
 	where we have used \eqref{conditiononT}.
 	
 	Thus, we can iterate the argument on successive intervals $[T, 2T]$, $[2T, 3T]$, and so on. By induction, we construct a unique global solution to \eqref{cauchyproblemindefinite}. Furthermore, at each step, the exponential decay estimate \eqref{ineqexpstabilityindefinite} is preserved, ensuring that the global solution satisfies this decay property.
 \end{proof}
 \section*{Acknowledgments.} This work was started
 when A. Moumni was visiting the University of L'Aquila. A. Moumni thanks the MAECI (Ministry of Foreign Affairs
 and International Cooperation, Italy) for funding that greatly facilitated scientific collaborations between
 Moulay Ismail University of Meknes (Morocco) and University of L'Aquila (Italy).

C. Pignotti is a member of Gruppo Nazionale per l’Analisi Matematica,
la Probabilità e le loro Applicazioni (GNAMPA) of the Istituto Nazionale di
Alta Matematica (INdAM).  She is partially
supported by PRIN 2022  (2022238YY5) {\it Optimal control problems: analysis,
approximation and applications}, PRIN-PNRR 2022 (P20225SP98) {\it Some mathematical approaches to climate change and its impacts}, and by INdAM GNAMPA Project {\it Optimal Control and Machine Learning} (CUP E5324001950001).
 
\section*{Conflicts of Interest}
The authors declare that they have no conflict of interest.
	
	%\section*{Data availability}
	%Data sharing not applicable to this article as no datasets were generated or %analysed during the current study.
	
	%\section*{Statements and Declarations}
	%The authors have no relevant financial or non-financial interests to disclose.

	%%%%%%%%%%%%%%%%%%%%%%%%%%%%%%%%%%%%%%%%%%%%%%%%%%%%%%%%%%%%%%%%%%%%%%%%%%%%%%%%%%
	\bibliographystyle{unsrt}
	\bibliography{MPST-bibliography-v}% common bib file
	
	%=================================
	%%%%%%%%%%%%%%%%%%%%%%%%%%%%%%%%%%%

\end{document}